\documentclass{amsart}

\usepackage{latexsym,amsmath,amssymb,amsfonts,amscd,graphics,appendix,amsxtra}
\usepackage[mathscr]{eucal}
\usepackage[all]{xypic}
\usepackage[normalem]{ulem}
\usepackage{hyperref}

\numberwithin{equation}{section}
\theoremstyle{plain}

\newtheorem*{mainthm}{Main Theorem}

\newtheorem{lemma}[equation]{Lemma}

\newtheorem{proposition}[equation]{Proposition}
\newtheorem{corollary}[equation]{Corollary}

\newtheorem{cordef}[equation]{Corollary/Definition}

\theoremstyle{remark}
\newtheorem{remark}[equation]{Remark}
\newtheorem*{Remark}{Remark}

\theoremstyle{definition}
\newtheorem{definition}[equation]{Definition}

\newtheorem{example}[equation]{Example}

\def\Id{\operatorname{Id}}

\def\End{\operatorname{End}}
\def\Hom{\operatorname{Hom}}

\def\Sym{\operatorname{Sym}}

\newcommand{\bZ}{\mathbb{Z}}
\newcommand{\bQ}{\mathbb{Q}}
\newcommand{\bR}{\mathbb{R}}
\newcommand{\bC}{\mathbb{C}}

\newcommand{\Aut}{\mathrm{Aut}}

\newcommand{\Res}{\mathrm{Res}}

\newcommand{\id}{\mathrm{id}}

\newcommand{\bS}{\mathbb{S}}     

\newcommand{\SL}{\mathrm{SL}}
\newcommand{\SU}{\mathrm{SU}}

\newcommand{\GL}{\mathrm{GL}}

\newcommand{\calD}{\mathcal{D}} 
\newcommand{\calH}{\mathcal{H}} 

\newcommand{\calA}{\mathcal{A}}

\newcommand{\calV}{\mathcal{V}}

\begin{document}
 \title[]{A realization for a $\bQ$-Hermitian variation of Hodge structure of Calabi-Yau type with real multiplication} 
\author[Z. Zhang]{Zheng Zhang} 
\address{Department of Mathematics \\ Stony Brook University \\ Stony Brook NY 11794}
\email{zzhang@math.sunysb.edu}
\thanks{The author was partially supported by NSF grants DMS-1254812 and DMS-1361143 (PI: R. Laza).}
\date{\today}

\bibliographystyle{alpha}
\maketitle

\begin{abstract}
We show that the $\bQ$-descents of the canonical $\bR$-variation of Hodge structure of Calabi-Yau type over a tube domain of type $A$ can be realized as sub-variations of Hodge structure of certain $\bQ$-variations of Hodge structure which are naturally associated to abelian varieties of (generalized) Weil type. 
\end{abstract}

\section*{Introduction}
It is an interesting problem to construct variations of Hodge structure (VHS) of geometric origin over Hermitian symmetric domains. The case of abelian varieties was studied by Shimura, Mumford and their collaborators during 1960's. For instance, abelian varieties with PEL-structures (polarization, endomorphism ring and level structure) are parametrized by certain arithmetic quotients of Hermitian symmetric domains (i.e. connected Shimura varieties). Here we are interested in the case of Calabi-Yau manifolds. Known examples of Calabi-Yau manifolds parameterized by Hermitian symmetric domains include those constructed in \cite{V}, \cite{Bor}, \cite{Rohde}, \cite{Rohde_2} and \cite{vGG}. 

Abstractly, we consider unconstrained VHS of Calabi-Yau (CY) type over Hermitian symmetric domains (called \textit{Hermitian VHS of CY type}), which are induced by equivariant, holomorphic, horizontal embeddings of Hermitian symmetric domains into the classifying spaces of Hodge structures of CY type. These have been studied by Gross \cite{Gro} (the tube domain case) and by Sheng and Zuo \cite{ZS} (the non-tube domain case). Based on their work, Friedman and Laza \cite{FL} classify Hermitian $\bR$-VHS of CY type. Specifically, over every irreducible Hermitian symmetric domain $\calD$, there exists a canonical $\bR$-VHS of CY type which corresponds to the minimal cominuscule representation (c.f. \cite{FL} Definition $2.5$) associated to the domain; any other irreducible CY Hermitian $\bR$-VHS on $\calD$ can be obtained from the canonical one via some standard operations (c.f. \cite{FL} Theorem $2.22$). Some partial analysis of Hermitian $\bQ$-VHS of CY type has also been done in \cite{FL} and \cite{FL_2}. In particular, Hermitian $\bQ$-VHS of CY $3$-fold type whose generic endomorphism algebras are totally real fields have been studied and classified (under a natural assumption) in \cite{FL} Theorem $3.18$ and \cite{FL_2}. 
 
It is natural to investigate the possibility of constructing Hermitian VHS of CY type from families of abelian varieties. In this short note, we give a motivic realization of the $\bQ$-descents of the canonical $\bR$-VHS of CY $n$-fold type over the Hermitian symmetric domain $(A_{2n-1}, \alpha_n)$ (and even more generally we allow real multiplication by an arbitrary totally real field). More precisely, we show that these $\bQ$-VHS of CY type can be realized as sub-VHS of the natural $\bQ$-VHS of weight $n$ associated to families of abelian varieties of (generalized) Weil type (see Corollary/Definition \ref{Weil}). 

\begin{mainthm}
Let $n$ be a positive integer. Let $E$ be a CM field with a totally real subfield $E_0$,  and set $d = [E_0 : \bQ]$. Also let $\calH$ be a product of $d$ irreducible Hermitian symmetric domains of type $A$: $\calH = (A_{2n-1}, \alpha_n) \times (A_{2n-1}, \alpha_{p_2}) \times \cdots \times (A_{2n-1}, \alpha_{p_d})$ with $p_i \equiv n$ (mod $2$) and $p_i < n$ for $2 \leq i \leq d$.  
\begin{itemize}
\item[(a)] There exists a family of abelian varieties of generalized Weil type (see Corollary/Definition \ref{Weil}) $\pi: \calA \rightarrow \calH$ over $\calH$ with the generic endomorphism algebra containing $E$.
\item[(b)] The VHS $R^n\pi_*\bQ$ contains an irreducible Hermitian $\bQ$-VHS of CY type whose generic endomorphism algebra is $E_0$. 
\end{itemize}
\end{mainthm}

This theorem generalizes the results in \cite{Lom} and \cite{CaFl}. Specifically, the case $E = \bQ(\sqrt{-r})$ ($r  \in \bZ^+$) and $n=2$ is essentially Theorem $3.8.1$ and Corollary $3.9$ of \cite{Lom}; the case $E = \bQ(\sqrt{-3})$ and $n=3$ is discussed in Section $3$ of \cite{CaFl}. 

\begin{Remark}
Following \cite{FL}, we say a Hermitian $\bR$-VHS of CY $3$-fold type is \textit{primitive} if the associated Hermitian symmetric domain is an irreducible tube domain of real rank $3$, and the Hermitian VHS of CY type is the canonical one over the domain. There are four primitive Hermitian VHS of CY $3$-fold type, which are listed in Corollary $2.29$ of op. cit.. The following are some results concerning their motivic realizations. 

The case $(C_3, \alpha_3)$ is classical and well-known. One can simply take the middle cohomology of an abelian threefold, which will contain a Hodge structure of CY type. At the other extreme, Deligne showed that there is no VHS of abelian variety type over $(E_7, \alpha_7)$. Thus the canonical VHS of CY type over $(E_7, \alpha_7)$ can not come from VHS of abelian variety type. The case $(D_6, \alpha_6)$ is much more subtle. On one hand, the Hermitian $\bQ$-VHS of CY type constructed in \cite{FL} Theorem $3.18$ and \cite{FL_2} is associated with Hermitian symmetric domains of mixed type $D$ (see Page $532$ of \cite {M}) in the presence of non-trivial real multiplication, over which there is no irreducible VHS of abelian variety type. On the other hand, assuming there is no non-trivial real multiplication, similar methods as in this note (but more representation theory) can be applied to realize the $\bQ$-descents of the unique irreducible Hermitian $\bR$-VHS of CY type contained in $\Sym^2 \calV$, where $\calV$ is the canonical CY $\bR$-VHS over the domain $(D_6, \alpha_6)$. 
\end{Remark}

After reviewing some background material on Hermitian VHS and Hodge representations in Section \ref{Pre}, we construct in Section \ref{AV} families of abelian varieties of generalized Weil type whose generic endomorphism algebras must contain a CM field $E$. This is a generalization of abelian varieties of Weil type (see for example \cite{vG} Definition $4.9$), where $E = \bQ(\sqrt{-r})$. Such families have been constructed in \cite{Sabelianvar} (see also Section $9.6$ of \cite{BLabelianvar}), but we shall reconstruct them from the perspective of Hodge representations (c.f. \cite{GGK} Chapter IV and \cite{M} Chapter $10$). In Section \ref{CY}, we prove Part (b) of the Main Theorem. One key ingredient is a proof of rationality of a certain representation using the ideas from Section $3$ of \cite{FL}.

\subsection*{Acknowledgement} The author is grateful to his advisor Radu Laza for introducing him to this problem, for encouragements and helpful suggestions. He would also like to thank Matthew Young for reading the draft and useful comments. 

\section{Preliminaries} \label{Pre}
In this section, we review some background on Hermitian VHS, especially in the case of abelian variety type (Hodge structures of level $1$) and CY type (effective Hodge structures of weight $k$ with $\dim H^{k,0} = 1$). 

Let $\calD = G(\bR)/K$ be an irreducible Hermitian symmetric domain, where $G$ is the almost simple and simply connected $\bR$-algebraic group associated to $\calD$. We first recall that irreducible Hermitian symmetric domains are classified by the root system $R$ of $G({\bC})$ together with a special root $\alpha_s$ (see for example \cite{M} Page $479$). In particular, an irreducible Hermitian symmetric domain of type $A$ is isomorphic to $\SU(p,q)/\mathrm{S}(\mathrm{U}(p) \times \mathrm{U}(q))$ for some $p,q \in \bZ^+$ (N.B. the associated simply connected algebraic group is $\SU(p,q)$) and corresponds to the pair $(A_{p+q-1}, \alpha_p)$. By choosing a suitable arithmetic subgroup of $\mathrm{Hol}(\calD)^+$, one can also assume that the associated algebraic group $G$ is defined over $\bQ$.

Except for the classical cases such as principally polarized abelian varieties or K3 surfaces, it is in general hard to understand the image of period maps because of Griffiths transversality. However, if a closed horizontal subvariety of a classifying space for certain Hodge structures satisfy certain natural assumptions as in \cite{FL} Definitions $1.1$ and $1.2$, then it must be an unconstrained Mumford-Tate domain (see Page $12$ of \cite{GGK}) and hence a Hermitian symmetric domain (c.f. Theorem $1.4$ of \cite{FL}). Following Definition $2.1$ of \cite{FL}, we will call such a VHS a \textit{Hermitian VHS}. 

Following Deligne, to give a Hermitian $\bQ$-VHS, one must give a representation $\rho: G \rightarrow GL(V)$ defined over $\bQ$ and a compatible polarization $Q$ on $V$ with $\rho(V) \subseteq \Aut(V,Q)$. As explained in Step $4$ of (IV.A) in \cite{GGK}, a compatible polarization typically exists and is unique. Also, without loss of generality, we assume that $\rho$ is irreducible over $\bQ$. 

We recall that the necessary and sufficient conditions for $\rho: G \rightarrow GL(V)$ together with a reference point $\varphi: U(1) \rightarrow \bar{G}$ ($\bar{G} = G/Z(G)$ is the adjoint form of $G$) to give a Hermitian VHS are as follows: there exists a reductive algebraic group $M \subseteq GL(V)$ defined over $\bQ$ (the generic Mumford-Tate group of the VHS) and a morphism of algebraic groups $h: \bS \rightarrow M_{\bR} \subseteq GL(V_{\bR})$ ($\bS = \Res_{\bC/\bR} \mathbb{G}_m$) such that 
\begin{itemize}
\item[(1)] The morphism $h$ defines a Hodge structure on $V$; 
\item [(2)] The representation $\rho$ factors through $M$ and $\rho(G) = M_{\mathrm{der}}$;
\item [(3)] The induced map $\bar{h} : \bS/\mathbb{G}_m \rightarrow M_{\mathrm{ad},\bR} = \bar{G}$ is conjugate to $\varphi: U(1) \rightarrow \bar{G}$.
\end{itemize}

\begin{Remark} \leavevmode
\begin{itemize} 
\item[(1)] Following \cite{GGK}, we call $\rho$ a Hodge representation. 
\item[(2)] Sub-representations of $V$ correspond to sub-Hermitian VHS and operations on representations correspond to the same operations on Hermitian VHS.
\end{itemize}
\end{Remark}

Satake \cite{Satake} and Deligne \cite{Dshimura} (especially Table $1.3.9$) classified Hodge representations of abelian variety type (see also Chapter $10$ of \cite{M}). Based on the earlier work of Gross \cite{Gro} and Sheng and Zuo \cite{ZS}, Friedman and Laza \cite{FL} classify Hermitian $\bR$-VHS (or Hermitian $\bQ$-VHS which remain irreducible over $\bR$) of CY type. Over every irreducible Hermitian symmetric domain $\calD$, there exists a canonical $\bR$-VHS $\calV$ of CY type; any other irreducible CY Hermitian $\bR$-VHS on $\calD$ can be obtained from the canonical $\calV$ by taking the unique irreducible CY factor of $\Sym^{\bullet} \calV$ or, for non-tube domains, $\Sym^{\bullet} \calV\{-\frac{a}{2}\}$ ($a \in \bZ$, $\{\}$ denotes the half twist, see \cite{FL} Section $2.1.3$). To describe the canonical VHS of CY type, we set $(R, \alpha_s)$ to be the pair corresponding to the domain $\calD$ and $G$ to be the associated algebraic group. Also let $\rho: G \rightarrow \GL(V)$ be a representation defined over $\bQ$ such that $V_{\bR} := V \otimes_{\bQ} \bR$ is still an irreducible representation. For tube domains, if the representation $V_{\bC} := V_{\bR} \otimes_{\bR} \bC$ of $G(\bC)$ is irreducible (in other words, the representation $V_{\bR}$ is of real type, see for example Page $88$ of \cite{GGK}) and has highest weight $\varpi_{s}$ which is the fundamental weight corresponding to the special root $\alpha_s$, then $\rho$ gives rise to a Hermitian $\bQ$-VHS of CY type whose scalar extension to $\bR$ is the canonical Hermitian $\bR$-VHS of CY type. We refer the readers to \cite{FL} for the description of the canonical CY variation on non-tube domains. 

To study Hermitian $\bQ$-VHS, it is important to consider the generic endomorphism algebras (or the endomorphism algebras of the corresponding Hodge representations).  Using them and the following lemma, we obtain decompositions (over $\bR$) of the VHS. 

\begin{lemma} \label{decomposition}
Let $V$ be a $\bQ$-vector space and $F$ be a finite field extension of $\bQ$ with $F \subseteq \End(V)$.
Denote the Galois closure of $F$ by $\tilde{F}$. Also let $G = \mathrm{Gal}(\tilde{F}/\bQ)$ and $H = \mathrm{Gal}(\tilde{F}/F)$. Note that we have $H \subseteq G$ and $[G:H] = [F: \bQ]$. Then the following statements hold.
\begin{itemize}
\item[(1)] There exists an $1-1$ correspondence between $G/H$ and $\Hom(F, \tilde{F})$ given by $gH \mapsto (a \mapsto g(a) )$. 
\item[(2)] Let $d = [F: \bQ]$ and $\sigma_1, \cdots, \sigma_d$ be the $d$ elements in $\Hom(F, \tilde{F})$. For any finite field extension $k \supseteq \tilde{F}$, the space $V_k := V \otimes_{\bQ} k$ decomposes into eigenspaces of the $F$-action. In other words, we have $V_k = \bigoplus_{i=1}^d V_i$ with $V_i = V \otimes_{F, \sigma_i} k = \{v \in V_k \mid  a (v) = \sigma_i(a) v,\ \forall a \in F \}$. 
\end{itemize}
\end{lemma}
\begin{proof}
See for example Section $2.4$ of \cite{vG_real}.
\end{proof}

\section{Families of abelian varieties of generalized Weil type} \label{AV}
Let us start by recalling the definition of abelian varieties of Weil type following \cite{vG}, which is a crucial ingredient for the constructions in \cite{Lom} and \cite{CaFl}.
\begin{definition} \label{Weil AV}
An abelian variety of Weil type of dimension $2m$ consists of a pair $(X,\bQ(\sqrt{-r}))$ with $X$ an abelian $2m$-fold and $t: \bQ(\sqrt{-r}) \hookrightarrow \End(X) \otimes \bQ$ an imaginary quadratic field such that for all $x \in \bQ(\sqrt{-r})$ the endomorphism $t(x)$ has $m$ eigenvalues $x$ and $m$ eigenvalues $\bar{x}$ on $T_0X$.
\end{definition}

\begin{remark}
Given an abelian $2m$-fold $(X, \bQ(\sqrt{-r}))$ of Weil type, we can construct a $\bQ(\sqrt{-r})$-vector space $V$ of dimension $2m$, a Hermitian form $H: V \times V \rightarrow \bQ(\sqrt{-r})$ of signature $(m,m)$ (thus one obtains a Hodge representation of $\SU(V, H)$) and a lattice $\Lambda \subseteq V$ as in \cite{vG} $5.2$, $5.3$. Conversely, given $(V, H, \Lambda)$ as above, one can construct a family of abelian varieties of Weil type over the irreducible Hermitian symmetric domain $(A_{2m-1}, \alpha_m)$ (see $5.4 - 5.10$ of op. cit.). This is the most convenient viewpoint for our generalization.
\end{remark}

We will consider abelian varieties whose endomorphism algebra contains an arbitrary CM field. In what follows, we let $E$ be a CM field with a totally real subfield $E_0$, and set $d = [E_0: \bQ]$. Also let $I = \Hom(E_0, \bR)$. Clearly, the set $I$ contains $d$ elements which will be denoted by $\sigma_1, \cdots, \sigma_d$.

Let  $n$ be a positive integer, and choose for each $i$ ($1 \leq i \leq d$) positive integers $p_i, q_i$ with $p_i + q_i = 2n$. Now for every $\sigma_i \in I$, we consider the $\bR$-algebraic group $G_{i,\bR} = \SU(p_i, q_i)$ and the corresponding irreducible Hermitian symmetric domain $\calD_i = (A_{p_i+q_i-1}, \alpha_{p_i})$ of type $A$.

\begin{lemma} \label{G_1}
There exists an algebraic group $G$ defined over $E_0$ such that $G \otimes_{E_0, \sigma_i} \bR \cong G_{i,\bR} = \SU(p_i, q_i)$. 
\end{lemma}
\begin{proof}
This follows from \cite{M} Proposition $10.14$. For our purposes, we give an explicit construction of $G$ for a chosen imaginary quadratic extension $E_0 \subseteq E$. Without loss of generality, we assume that $p_i \geq p_{i+1}$ (hence $q_i \leq q_{i+1}$). By the approximation theorem (see for example Section $1.2$ of \cite{NT}), there exists an element $\delta_1 \in E_0$ such that $\sigma_1(\delta_1) > 0$ and $\sigma_i(\delta_1) < 0$ for $i >1$. Similarly, we choose $\delta_2, \cdots, \delta_d$ from $E_0$ such that $\sigma_i(\delta_i)$ is positive and $\sigma_i(\delta_j)$ is negative for $i \neq j$. Let us observe here that $(-1)^{i-1} \delta_1 \cdot \delta_2 \cdots \delta_i$ is positive under the first $i$ embeddings $\sigma_1, \sigma_2, \cdots, \sigma_i$, while negative under other embeddings $\sigma_{i+1}, \cdots, \sigma_d$.

Let $U$ be an $E$-vector space of dimension $2n$ together with an $E/E_0$-Hermitian form $h$ defined as follows. In a suitable basis, we define $h$ by a diagonal matrix with $q_1$ copies of $(-1)$'s,  $(q_{i+1} - q_i)$ copies of $[(-1)^{i-1}\delta_1\delta_2\cdots\delta_i]$'s for $1 \leq i \leq (d-1)$ and $p_d$ copies of $1$'s on the diagonal. (In particular, the number of elements is $q_1 + (q_2-q_1) + (q_3 - q_2) + \cdots + (q_d - q_{d-1}) + p_d = p_d + q_d = 2n$.) It follows from the construction that $h \otimes_{E_0, \sigma_i} \bR$ has signature $(p_i, q_i)$.

Let $G = \SU(U, h)$. It is not difficult to see that $G \otimes_{E_0, \sigma_i} \bR \cong \SU(U \otimes_{E_0, \sigma_i} \bR, h \otimes_{E_0, \sigma_i} \bR) \cong \SU(p_i, q_i)$.
\end{proof}

Let $U$, $h$ and $G$ be as in the proof of the previous lemma. Then we have a standard representation $\rho : G \rightarrow \GL(U)$. Now we set $G' = \Res_{E_0/\bQ} G$, $U' = \Res_{E_0/\bQ} U$ (N.B. $\dim_{\bQ} U' = 4nd$) and consider the representation 
\begin{equation} \label{HR}
\rho' = \Res_{E_0/\bQ} (\rho) : G' \rightarrow \GL(U').
\end{equation}
This is the representation that interests us.

\begin{proposition} \label{SR}
The representation $\rho'$ is a Hodge representation corresponding to a Hermitian $\bQ$-VHS of abelian variety type over $\calH = \prod_{i=1}^d \calD_i$ and gives a family of abelian varieties over $\calH$. 
\end{proposition}
\begin{proof}
Let $W$ be an irreducible summand of $U'_{\bC} := U' \otimes_{\bQ} \bC$. In general, we have $W \cong \bigotimes_{\sigma_i \in I} W_i$ where $W_i$ is an irreducible representation of $G_{i,\bC} = G_{i,\bR} \otimes_{\bR} \bC = \SL(2n,\bC)$. By Theorem $10.18$ and Summary $10.11$ of \cite{M}, it suffices to check the following conditions to show that the $\bQ$-representation $\rho'$ gives a Hermitian $\bQ$-VHS of abelian variety type over $\calH$.
\begin{itemize}
\item[(1)] For each irreducible summand $W \cong \bigotimes_{\sigma_i \in I} W_i$ of $U'_{\bC}$, there exists at most one $i$ such that $W_i$ is a non-trivial representation of $G_{i,\bC}$.
\item[(2)] Let $W_i$ be the non-trivial factor of $W \cong \bigotimes_{\sigma_i \in I} W_i$. If $\calD_i$ is $(A_{2n-1}, \alpha_{p_i})$ with $2 \leq p_i \leq 2n-2$, then $W_i$ has highest weight $\varpi_1$ or $\varpi_{2n-1}$. If $\calD_i$ is $(A_{2n-1}, \alpha_1)$ or $(A_{2n-1},  \alpha_{2n-1})$ (i.e. $p_i =1$ or $2n-1$), then $W_i$ can have arbitrary highest weight.
\end{itemize}

We now verify these conditions by analyzing the irreducible summands of $U'_{\bC}$. By Lemma \ref{decomposition}, the space $U'_{\bR} := U' \otimes_{\bQ} \bR$ decomposes into eigenspaces of the $E_0$-action. Specifically,  we have $U'_{\bR} = \bigoplus_{i=1}^d U_{i,\bR}$ with $U_{i,\bR} = U \otimes_{E_0, \sigma_i} \bR$. By the construction we know that $U_{i,\bR}$ is the standard representation of $G_{i,\bR} = \SU(p_i, q_i)$. It is then not difficult to see that $U_{i, \bR}$ is of complex type, i.e. $U_{i, \bR} \otimes_{\bR} \bC \cong U_{i,+} \oplus U_{i,-}$, where $U_{i,+}$ (resp. $U_{i,-}$) is the irreducible representation of $\SL(2n, \bC)$ with highest weight $\varpi_1$ (resp. $\varpi_{2n-1}$). Since $U_{i,+}$ and $U_{i,-}$ ($1 \leq i\leq d$) are all the irreducible summands of $U'_{\bC}$, Conditions $(1)$ and $(2)$ hold.

Finally, this Hermitian $\bQ$-VHS of abelian variety type gives a family of abelian varieties up to choices of underlying integral structures. (c.f. Theorem $11.8$ of op. cit.)
\end{proof}

We now define a family of abelian varieties of generalized Weil type. Using this, the first part of the Main Theorem will be clear. 

\begin{cordef} \label{Weil}
Let $\pi: \calA \rightarrow \calH = \prod^d_{i=1} \calD_i$ be a family of abelian varieties. We say $\pi$ is of generalized Weil type if the natural $\bQ$-VHS $R^1\pi_*\bQ$ associated to $\pi$ corresponds to the Hodge representation $\rho'$ defined in (\ref{HR}), and the domains $\calD_i = (A_{p_i +q_i -1}, \alpha_{p_i})$ satisfies the following conditions:
\begin{itemize}
\item[(1)] $p_1 = q_1 = n$;
\item[(2)] For $2 \leq i \leq d$, $ 0 < p_i < q_i$, $p_i + q_i = 2n$ and $p_i \equiv n$ (mod $2$).
\end{itemize}
\end{cordef}

\begin{remark}
Later we will see that Condition $(1)$ is needed to obtain a VHS of CY $n$-fold type over $\calD_1$. For $2 \leq i \leq d$, the requirement that $p_i < q_i$ (and hence $p_i < n$) is made to have a VHS over $\calD_i$ of smaller Hodge level. The condition that $p_i \equiv n$ (mod $2$) is the most natural one for us to put so that certain representations on $E$-vector spaces are actually defined over $E_0$ (see Lemma \ref{R type}) which is crucial to our construction. 
\end{remark}

\begin{remark}
Let $\pi$ be a family of abelian varieties of generalized Weil type, then $\bigwedge^n_{\bQ} \rho': G' \rightarrow \GL(\bigwedge^n_{\bQ} U')$ corresponds to the natural VHS associated to $\pi$ with local system $R^n\pi_*\bQ$.
\end{remark}

\section{Constructions of Hermitian VHS of CY type} \label{CY}
We will complete the proof of the Main Theorem in this section. As in the previous sections, $E$ is a CM field with a totally real subfield $E_0$. The Hodge representation $\rho'$ in \eqref{HR} gives a family of abelian varieties of generalized Weil type $\pi : \calA \rightarrow \calH$. Other notations also remain as in Section \ref{AV}. 

The basic idea is to consider the $G$-representation $\bigwedge^n_{E} U$, which is essentially of CY type, but defined over $E$. Before proving the Main Theorem, we show in Corollary \ref{sub-rep} that $\Res_{E/\bQ} \bigwedge^n_{E} U$ is naturally a sub-representation of $\bigwedge^n_{\bQ} U'$, and in Lemma \ref{R type} that the representation $\bigwedge^n_{E} U$ is defined over $E_0$. 

To start, we prove the following lemma in a slightly general situation.

\begin{lemma} \label{res}
Let $L \subseteq F \subseteq K$ be finite field extensions, and $W$ be a finite dimensional $K$-vector space. Also choose a positive integer $l$ with $1 \leq l \leq \dim_K W$.  
\begin{itemize}
\item[(1)] There exists a natural injection between $F$-vector spaces: $\Res_{K/F} (\bigwedge^l_{K} W) \hookrightarrow \bigwedge^l_{F} (\Res_{K/F} W)$. 
\item[(2)] There is a canonical isomorphism between $\Res_{F/L} (\Res_{K/F} W)$ and $\Res_{K/L} W$.
\end{itemize}
\end{lemma}	

\begin{proof} 
We first prove Part $(1)$. First of all, for any finite dimensional $K$-vector space $M$ there is a natural isomorphism $\Res_{K/F} \Hom_K(M, K) \rightarrow \Hom_F(\Res_{K/F} M, F)$ given by $f \mapsto \mathrm{Tr} \circ f$, where $\mathrm{Tr} : K \rightarrow F$ is defined by $\mathrm{Tr}(z) = \sum_{g \in \mathrm{Gal}(K/F)} g(z)$. To see this, note firstly that  $\mathrm{Tr} : K \rightarrow F$ is $F$-linear, so the map $f \mapsto \mathrm{Tr} \circ f$ is well-defined and $F$-linear. Next, this is an injective map. Indeed, if $\mathrm{Tr} \circ f \equiv 0$ for a $K$-linear map $f: M \rightarrow K$, and if there exists $v \in M$, such that $f(v) \neq 0$, then for any $c \in K$, we have $\mathrm{Tr}(c) = \mathrm{Tr} (f(\frac{c}{f(v)} \cdot v)) = 0$ which is a contradiction. Now the claim follows from an easy dimension count. 

Secondly, one can identify $\bigwedge_k^l N^*$ with $(\bigwedge_k^l N)^*$ for any field $k$ and $k$-vector space $N$, where $N^* = \Hom_k(N,k)$ and similarly for $(\bigwedge_k^l N)^*$. 

Furthermore, we observe that for any finite dimensional $K$-vector space $M$, the natural map $\prod^l \Res_{K/F} M \twoheadrightarrow \Res_{K/F} \bigwedge^l_{K} M$ induces a surjection $\bigwedge^l_{F} \Res_{K/F} M \twoheadrightarrow \Res_{K/F} \bigwedge^l_{K} M$. In particular, we have the following natural surjection  
$$\bigwedge \limits^l_F \Res_{K/F} W^* \twoheadrightarrow \Res_{K/F} \bigwedge \limits^l_KW^*,$$
where $W^* = \Hom_K(W,K)$. Taking duals, we obtain a natural inclusion 
$$\Hom_F(\Res_{K/F} \bigwedge \limits^l_K W^*, F) \hookrightarrow \Hom_F(\bigwedge \limits^l_F \Res_{K/F} W^*, F).$$

Let us denote $\Res_{K/F}$ by $\Res$. Using the above observations, we get a chain of injective maps as follows.
$$\Res \bigwedge \limits^l_K W \cong \Res \bigwedge \limits^l_K \Hom_K(W^*, K) \cong \Res \Hom_K(\bigwedge \limits^l_K W^*, K) \cong \Hom_F(\Res\bigwedge \limits^l_K W^* , F) $$
$$ \hookrightarrow \Hom_F(\bigwedge \limits^l_F \Res W^*, F) \cong \bigwedge \limits^l_F \Hom_F(\Res W^*, F) \cong \bigwedge \limits^l_F \Res \Hom_K(W^*, K) \cong \bigwedge \limits^l_F \Res W.$$

As for Part $(2)$, the isomorphism is given by the identity map. It is clear that the map is $L$-linear and both vector spaces have the same $L$-dimension.

\end{proof}

Moreover, when $W$ is a representation of some algebraic group $H$, we note the following lemma.
\begin{lemma}
Notations as in the previous lemma. Suppose $W$ is a $K$-representation of an algebraic group $H$ defined over $F$, then the natural inclusion described in $(1)$ (resp. the natural isomorphism in $(2)$) of Lemma \ref{res} commutes with the action of $H$ (resp. of $\Res_{F/L} H$).
\end{lemma}

\begin{proof}
For $(1)$, one can check that each linear map in the chain commutes with the $H$-action. For example, let us check this for the map $\phi: \Res_{K/F} \Hom_K(M, K) \rightarrow \Hom_F(\Res_{K/F} M, F)$, $f \mapsto \mathrm{Tr} \circ f$, where $M$ is a $K$-representation of $H$. Indeed, for any $g \in H$, $(\phi(g \cdot f))(v) = (\mathrm{Tr} \circ ( g \cdot f))(v) = \mathrm{Tr} ((g \cdot f)(v)) = \mathrm{Tr}(f (g^{-1} \cdot v)) = \phi(f)(g^{-1} \cdot v) = (g \cdot \phi( f ))(v)$. It is not difficult to also check that $\phi \otimes_F {A}$ commutes with the action of $H(A)$ for any $F$-algebra $A$. Part $(2)$ is clear. 
\end{proof}

Inspired by Proposition $3.1$ of \cite{Lom}, we observe that in certain situations $\bigwedge^l_K W$ can be described as kernel of some endomorphism of $\bigwedge^l_F (\Res_{K/F} W)$ quite explicitly.

\begin{example} 
Let $F = \bQ$ and $K = \bQ(\sqrt{-r})$ for some positive integer $r$, and choose an element $\varphi \in K$ such that $\varphi^2 = -r$. Also suppose that $\dim_K W = 6$. By writing down the chain of injective maps in Lemma \ref{res} Part (a) (the composition will be denoted by $i$) explicitly, one can show that the image of $ v_1 \wedge_{K} v_2 \wedge_{K} v_3 \in \bigwedge^3_K W$ is $\frac{1}{4} [v_1 \wedge v_2 \wedge v_3+\frac{1}{\varphi^2} \cdot (\varphi v_1 \wedge \varphi v_2 \wedge v_3 + \varphi v_1 \wedge  v_2 \wedge \varphi v_3 + v_1 \wedge \varphi v_2 \wedge \varphi v_3)]$, where $\wedge$ means $\wedge_\bQ$. Now we define $\varphi_3: \bigwedge^3_{\bQ}(\Res_{K/\bQ} W) \rightarrow \bigwedge^3_{\bQ} (\Res_{K/\bQ} W)$ to be the map $v_1 \wedge v_2 \wedge v_3 \mapsto (\varphi v_1 \wedge \varphi v_2 \wedge v_3 + \varphi v_1 \wedge  v_2 \wedge \varphi v_3 + v_1 \wedge \varphi v_2 \wedge \varphi v_3)$. The map $\varphi_3$ is well-defined and it is straightforward to check that $i(\bigwedge^3_K W) = \ker(\varphi_3 - 3\varphi^2 \cdot \id)$. More generally, for $W$ with $\dim_K W = 2m$ and $v_1, \cdots, v_m \in W$,  we pick up two elements from $\{v_1, \cdots, v_m\}$ and multiply them by $\varphi$ in the expression $v_1 \wedge \cdots \wedge v_m$. The sum of such ${m \choose 2}$ elements of $\bigwedge^m_{\bQ}(\Res_{K/\bQ} W)$ are defined to be $\varphi_m(v_1 \wedge \cdots \wedge v_m)$. In a similar way one can show that $i(\bigwedge^m_K W) = \ker(\varphi_m - \binom{m}{2} \cdot \varphi^2 \cdot \id)$.
\end{example}

Now let us return to the Hodge representation $\rho'$ in \eqref{HR}.
\begin{corollary} \label{sub-rep}
With notations as in Section \ref{AV}, $\Res_{E/\bQ} (\bigwedge^n_E U)$ is naturally a $G'$-subrepresentation of $\bigwedge^n_{\bQ}  U'$.
\end{corollary}
\begin{proof}
Applying the previous two lemmas for $l = n$, $L =\bQ$, $F = E_0$, $K =E$, $W = U$ and $H = G$, we get $$\Res_{E/\bQ} (\bigwedge^n_E U) \cong \Res_{E_0/\bQ} (\Res_{E/ E_0} \bigwedge^n_E U)  \hookrightarrow \Res_{E_0/\bQ} (\bigwedge^n_{E_0} \Res_{E/E_0} U).$$ Similarly, when $l = n$, $F = \bQ$, $K =E_0$, $W = \Res_{E/E_0} U$ and $H = G'$, it is not hard to see that there is an injective map $\Res_{E_0/\bQ} (\bigwedge^n_{E_0} \Res_{E/E_0} U) \hookrightarrow \bigwedge^n_{\bQ}  U'$.
\end{proof}	

As we observed, the representation $G' \rightarrow \GL(\bigwedge^n_{\bQ}  U')$ corresponds to the Hermitian VHS over $\calH$ with local system $R^n\pi_* \bQ$. So by the previous corollary the representation $\Res_{E/\bQ} (\bigwedge^n_E U)$ corresponds to a sub-VHS of $R^n\pi_* \bQ$.

We now verify the rationality of the $G$-representation $\bigwedge^n_E U$ following Section $3.5$ of \cite{FL}. Recall that $h$ is the Hermitian form on the $2n$-dimensional $E$-vector space $U$ as constructed in Lemma \ref{G_1}.
\begin{lemma} \label{R type}
The $G$-representation $\bigwedge^n_E U$ is defined over $E_0$. In other words, there exists a sub-representation on an $E_0$-vector space $W_0 \subseteq \Res_{E/E_0} (\bigwedge^n_E U)$, such that $W_0 \otimes_{E_0} E \cong \bigwedge^n_E U$. 
\end{lemma} 
\begin{proof}
Let us consider $\mathcal{A} := \End_{E_0[G]}(\Res_{E/E_0}(\bigwedge^n_E U))$. By Proposition $3.22$(iv) of \cite{FL} it suffices to show that $\mathcal{A}$ is the matrix algebra $\mathbb{M}_2(E_0)$. We first construct an operator $\star \in \mathcal{A}$. Note that there are two natural pairings
$$ \wedge:  \bigwedge \limits^n_E U \times \bigwedge \limits^n_E U \rightarrow \bigwedge \limits^{2n}_E U \cong E$$
and $$\wedge^n h: \bigwedge \limits^n_E U \times \bigwedge \limits^n_E U \rightarrow  E,$$
where $\wedge^n h$ is defined by $(\wedge^n h)(v_1 \wedge \cdots \wedge v_n, u_1 \wedge \cdots \wedge u_n) = \det(h(v_i, u_j))$. They give an $E$-linear isomorphism $\tau_1: \bigwedge^n_E U \rightarrow \bigwedge^n_E U^*$ and an $E$-conjugate-linear isomorphism $\tau_2: \bigwedge^n_E U \rightarrow \bigwedge^n_E U^*$ respectively. The operator $\star$ is then defined by $\star = \tau_1^{-1} \circ \tau_2$. It is not hard to verify that $\star \in \mathcal{A}$. Moreover, direct computation shows that $\star \circ \star = (-1)^n \mathrm{disc}(h)$, where $\mathrm{disc}(h)$ is the discriminant of $h$ which is defined by $\mathrm{disc}(h) \equiv \det(h) \in E_0^*/\mathrm{Nm}_{E/E_0}(E^*)$.

Suppose $E = E_0(\sqrt{-e})$, then by Proposition $3.22$(i) of op. cit. $\mathcal{A}$ is a quaternion algebra over $E_0$ with a quaternion basis $\mathbf{1} = \Id$, $\mathbf{i} = \sqrt{-e}\Id$, $\mathbf{j} = \star$, $\mathbf{k} = \mathbf{i} \mathbf{j}$. By a standard fact on quaternion algebras (see for example Proposition $1.1.7$ of \cite{Quaternion}) it is now enough to show that there is an element $c$ in $E$ such that  $\star \circ \star = \mathrm{Nm}_{E/E_0}(c)$. Using Lemma \ref{G_1}, we have $\star \circ \star = (-1)^n \mathrm{disc}(h) = (-1)^n \cdot (-1)^n \cdot \delta_1^{q_2-n} \cdot (-\delta_1\delta_2)^{q_3-q_2} \cdots [(-1)^d \delta_1\delta_2 \cdots \delta_{d-1}]^{q_d-q_{d-1}}$. By Corollary/Definition \ref{Weil} $(2)$, $q_{i+1}-q_i$ is even for $1 \leq i \leq d-1$, so $(-1)^n \mathrm{disc}(h)$ is a square in $E_0$, and hence is of the required form. 
\end{proof}

Let us now prove Part (b) of the Main Theorem. The notations are the same as in Section \ref{AV}. For example, $U_{i,\bR} = U \otimes_{E_0,\sigma_i} \bR$.

\begin{proof}[Proof of the Main Theorem]
First note that $\Res_{E_0/\bQ} W_0$ is a sub-representation of $\Res_{E_0/\bQ} (\Res_{E/E_0} \bigwedge^n_{E} U) \cong \Res_{E/\bQ} (\bigwedge^n_E U) \subseteq \bigwedge^n_{\bQ} U'$. We claim that $\Res_{E_0/\bQ} W_0$ corresponds to a $\bQ$-Hermitian VHS of CY type with real multiplication by $E_0$. To see it, let us consider the decomposition $(\Res_{E_0/\bQ} W_0) \otimes_{\bQ} \bR \cong \bigoplus_{i=1}^d (W_0 \otimes_{E_0, \sigma_i} \bR)$. For $i=1$, the representation $W^1_{0, \bR} := W_0 \otimes_{E_0, \sigma_1} \bR$ gives rise to a Hermitian $\bR$-VHS of CY $n$-fold type over $\calD_1 = (A_{2n-1}, \alpha_n)$. Indeed, by Lemma \ref{R type} we have that $W^1_{0, \bR} \otimes_{\bR} \bC \cong (W_0 \otimes_{E_0} E) \otimes_{E_0, \sigma_1} \bR \cong (\bigwedge^n_E U) \otimes_{E_0, \sigma_1} \bR \cong \bigwedge^n_{\bC} U_{1,\bR}$ which has highest weight $\varpi_n$ (note that $U_{1,\bR}$ is isomorphic to the standard representation of $\SL(2n, \bC)$), and then one can apply Lemma $2.23$ of \cite{FL} or Section $3$ of \cite{Gro} to see that $W^1_{0, \bR}$ corresponds to the canonical CY $\bR$-VHS on $\calD_1$. Similarly, for $W^i_{0, \bR} := W_0 \otimes_{E_0, \sigma_i} \bR$ with $2 \leq i \leq d$, we have that $W^i_{0,\bR} \otimes_{\bR} \bC \cong \bigwedge^n_{\bC} U_{i,\bR}$ which is an irreducible representation of $\SL(2n, \bC)$ with highest weight $\varpi_n$. In particular, we see that $W^i_{0, \bR}$ is of real type as an irreducible representation of $G _{i, \bR} = \SU(p_i, q_i)$ (or one can apply \cite{GGK} (IV.E.4) to see that the conditions $p_i  \equiv n$ (mod $2$) imply this). Now by \cite{FL} $(2.12)$ the Hermitian VHS over $\calD_i = (A_{2n-1}, \alpha_{p_i})$ given by the representation $W^i_{0, \bR}$ has level $2\varpi_n(H_0)$, where $H_0$ is the differential of a reference point on $\calD_i$ (c.f. Section $2.1$ of \cite{FL}). It can be easily seen that $\varpi_n(H_0)$ equals the coefficient of $\alpha_{p_i}$ in the expression of $\varpi_n$ (for the root system $A_{2n-1}$), and hence $2\varpi_n(H_0) = p_i$. Since $p_i < n$ for $i \geq 2$, only the CY piece $W^1_{0, \bR}$ contributes to the Hodge number $h^{n,0}$ and hence $\Res_{E_0/\bQ} W_0$ gives a $\bQ$-Hermitian VHS of CY type. To complete the proof, let us note that $E_0 \subseteq \End_{G'}(\Res_{E_0/\bQ} W_0)$. Moreover, using the first remark in the proof of \cite{FL} Proposition $3.7$ and the irreducibility of the representations $W^i_{0, \bR}$ ($1 \leq i \leq d$), we know that $\End_{G'}(\Res_{E_0/\bQ} W_0)$ has dimension $d$, and hence $E_0$ is the generic endomorphism algebra and the corresponding CY Hermitian $\bQ$-VHS is irreducible.
\end{proof}

\begin{remark}
When $E = \bQ(\sqrt{-r})$ and $E_0 = \bQ$, $\pi: \calA \rightarrow \calH = (A_{2n-1}, \alpha_n)$ is a family of abelian varieties of Weil type. Furthermore, the proof of the Main theorem shows that $R^n\pi_*\bQ$ contains a sub-VHS of CY type whose scalar extension to $\bR$ is isomorphic to the canonical CY $\bR$-VHS over $\calH$. This generalizes the constructions in \cite{Lom} ($n=2$) and \cite{CaFl} ($n=3$ and $E = \bQ(\sqrt{-3})$). 
\end{remark}

\begin{remark} 
If $n$ is odd and $p_i=1$ for $i \geq 2$, then by Summary $10.8$ of \cite{M} every representation $W^i_{0, \bR} = W_0 \otimes_{E_0, \sigma_i} \bR$ ($i \geq 2$) determines a Hermitian VHS of abelian variety type over $(A_{2n-1}, \alpha_1)$. In particular, when $n=3$ we obtain a motivic realization of the CY Hermitian $\bQ$-VHS constructed in \cite{FL} Theorem $3.18$(ii)(a).
\end{remark}

\begin{remark}
We give a hint on how to compute the Hodge numbers of the Hermitian $\bR$-VHS over $(A_{2n-1}, \alpha_{p_i})$ corresponding to the representation $W^i_{0, \bR}$ for a fixed $i$ ($1 \leq i \leq d$). We shall use the notations in \cite{FH_rep}. Specifically, we let $L_j$ ($1 \leq j \leq 2n$) be those defined in Section $15.1$ of op. cit. (N.B. $L_1 + \cdots + L_{2n} = 0$). Then a set of simple roots of $A_{2n-1}$ are given by $\alpha_j = L_j - L_{j+1}$ with $1 \leq j \leq 2n-1$ and the corresponding fundamental weights are $\varpi_j = L_1 + \cdots + L_j$ with $1 \leq j \leq 2n-1$. We have seen in the proof of the Main theorem that $W^i_{0, \bR} \otimes_{\bR} \bC$ has highest weight $\varpi_n$ and thus the corresponding character is $\sum e(\xi)$, the sum over all $\xi$ that are sums of $n$ different $L_j$ with $1 \leq j \leq 2n$ (c.f. Page $377$ of op. cit.). It is straightforward to check that (using for example $L_j = \varpi_j - \varpi_{j-1}$ and \cite{Humphreys} Section $13.1$ Table $1$)
$$\text{The coefficient of} \ \alpha_{p_i} \ \text{in} \ L_j  = 
\begin{cases}
\frac{2n - p_i}{2n} &\text{if}  \ \ 1 \leq j \leq p_i, \\
-\frac{p_i}{2n} &\text{if} \ \ p_i +1 \leq j \leq 2n.
\end{cases}$$ 
Now one can see that if $\xi$ contains exactly $s$ ($0 \leq s \leq p_i$) elements of $\{L_1, \cdots, L_{p_i}\}$, then $2\xi(H_0) = 2s-p_i$, and by \cite{FL} $(2.12)$ the weight space of $\xi$ (which is $1$-dimensional) contributes to the Hodge number $h^{s, p_i-s}$. To conclude, the Hermitian VHS corresponding to $W^i_{0, \bR}$ has level $p_i$ and the Hodge numbers (assuming the weight is $p_i$) are
$$h^{s, p_i-s} = {p_i\choose s}{2n-p_i \choose n-s}, \ \ \ 0 \leq s \leq p_i.$$ 
\end{remark}

\bibliography{ref_z}
\end{document}